\documentclass[11pt,a4paper]{article}
\usepackage[left=3.00cm, right=3.00cm, bottom=2.00cm, top=3.00cm]{geometry}
\usepackage{amsmath, amsfonts, amsthm, amssymb, mathrsfs}
\usepackage{graphicx}
\usepackage{silence}
\DeclareUnicodeCharacter{2212}{\ensuremath{-}}
\newtheorem{teorema}{Theorem}
\newtheorem{lemma}[teorema]{Lemma}

\title{Determining the Locating Rainbow Connection Number of Vertex-Transitive Graphs}
\author{Ariestha Widyastuty Bustan$^{1,3}$\\ A.N.M. Salman$^{2}$ \\ Pritta Etriana Putri$^{2}$}
\date{}

\begin{document}

\maketitle

\begin{center}
	\textit{$^1$Doctoral Program of Mathematics, Faculty of Mathematics and Natural Sciences, Institut Teknologi Bandung, Bandung, Indonesia} \\
	\textit{$^2$Combinatorial Mathematics Research Group, Faculty of Mathematics and Natural Sciences, Institut Teknologi Bandung, Bandung, Indonesia} \\
	\textit{$^3$Mathematics Department, Faculty of Mathematics and Natural Sciences, Universitas Pasifik Morotai, Kabupaten Pulau Morotai, Indonesia} \\
\end{center}
\begin{center}
\texttt{30119004@mahasiswa.itb.ac.id} 
\end{center}

\begin{abstract}
The locating rainbow connection number of a graph is defined as the minimum number of colors required to color vertices such that every two vertices there exists a rainbow vertex path and every vertex has a distinct rainbow code. This rainbow code signifies a distance between vertices within a given set of colors in a graph. This paper aims to determine the locating rainbow connection number for vertex-transitive graphs. Three main theorems are derived, focusing on the locating rainbow connection number for some vertex-transitive graphs.
\end{abstract}

\section{Introduction}
In the theory of graph domains, there is a concept known as chromatic coloring. According to this concept, with $G=(V(G),E(G))$ being a finite, undirected, and connected graph, the chromatic coloring of $G$ involves assigning colors to vertices in such a way that no two adjacent vertices share the same color. To describe the minimum number of colors required for such a chromatic coloring of $G$, we utilize the symbol $\chi(G)$, which represents the chromatic number of $G$.

In addition to the concept of locating chromatic number, Chartrand et al. \cite{chartrand2002locating} also introduced the concept of rainbow coloring in 2008. This concept was inspired by the undercover communication techniques employed by government agencies to guarantee the secure transfer of classified information, especially in the aftermath of the 9/11
attacks in 2001 \cite{ericksen2007matter}. Since then, this concept has been widely studied, involving a variety of graph operations and graph classes (e.g.,\cite{fitriani2022rainbow}, \cite{kumala2015rainbow}, \cite{nabila2015rainbow}, \cite{susanti1ab2020rainbow}, and \cite{umbara2023inverse}). Motivated by the concept of rainbow coloring, in 2010, Krivelevich and Yuster introduced rainbow vertex coloring of a graph \cite{krivelevich2010rainbow}. Following this, the rainbow vertex connection
number of several classes of graphs has been a focus of several studies (e.g., \cite{bustan2018rainbow} and \cite{simamora2015rainbow}.

Motivated by the concepts of rainbow vertex coloring and dimension partition, a concept that combines both, called locating rainbow coloring of a graph, was introduced in 2021 \cite{bustan2021locating}. For a natural number $k$, a coloring of the vertex set of $G$ is termed a \textit{rainbow vertex k-coloring} if there exists a function $c:V(G)\longrightarrow \{1,2,...,k\}$ such that, for any distinct pair of vertices $x,y\in V(G)$, there is \textit{rainbow vertex $x-y$-path}, whose internal vertices are assigned a different color. The \textit{rainbow vertex connection number} of $G$, denoted by $rvc(G)$, is the smallest positive integer $k$, so $G$ has a rainbow vertex $k$-coloring.  For $i \in \{1,2,...,k\}$, let $R_i$ denote the set of vertices that have the color $i$ and $\Pi=\{R_1,R_2,...,R_k \}$ be an ordered partition of $V(G)$. Thus, $rc_{\Pi} (v)=(d(v,R_1 ),d(v,R_2 ),...,d(v,R_k ))$, where $d(v,R_i )=\min\{ d(v,y):y\in R_i \}$ for every $i\in \{1,2,...,k\}$. Further, $rc_{\Pi} (v)$ is called the \textit{rainbow code} of $v$ of $G$ with respect to $\Pi$. If $rc_{\Pi} (v_j )\neq rc_{\Pi} (v_l )$ for distinct $j,l \in \{1,2,...,n\}$, then the coloring $c$ is known as a \textit{locating rainbow $k$-coloring} of $G$. The smallest positive integer $k$ for which a locating rainbow $k$-coloring exists in the graph $G$, denoted by $rvcl(G)$, represents the locating rainbow connection number of graph $G$ \cite{bustan2021locating}. It is important to note that every locating rainbow $k$-coloring of $G$ also serves as a rainbow vertex coloring of $G$, implying that
\begin{equation}
	rvc(G)\leq rvcl(G).
	\label{eq1}
\end{equation}

Several results regarding the $rvcl(G)$ can be found in \cite{bustan2021locating}, \cite{bustan2023locating}, and \cite{bustan1} with some required results in this paper as follows.

\begin{lemma}\cite{bustan2021locating}
	Let $c$ be a locating rainbow coloring of $G$. Let $u$ and $v$ be two distinct vertices of $G$. If $d(u,x)=d(v,x)$ for all $x \in V(G)-\{u,v\}$, then $c(u)\neq c(v)$.
	\label{lemma1}
\end{lemma}

\begin{lemma}\cite{bustan2023locating}
If $G$ is a simple connected graph of order $n \geq 3$ which contains a cycle, then  $rvcl(G)\geq 3$.
\label{lemamemuatliangkaran}
\end{lemma}

\begin{teorema}\cite{bustan2023locating}
Let $G$ be a connected graph of order $n \geq 3$. Then $rvcl(G)=n$ if and only if $G$ is isomorphic to complete graphs.
\label{lemakn}
\end{teorema}

\begin{teorema}\cite{bustan2021locating}
Let $diam(G)$ denote a diameter of $G$. If $G$ is a connected graph of order $n\geq 3 $ and $rvcl(G)=r$, then $n \leq r \times diam(G)^{r-1}$.
\label{theoremdiam}
\end{teorema} 

Lemma \ref{lemma1} concludes that if $c$ is a locating rainbow coloring of $G$, then two distinct vertices whose have the same distance to other vertices would not share the same color. Meanwhile, Theorem \ref{theoremdiam} demonstrates the assignment of a specific value for $rvcl(G)$ to determine the maximum number of vertices in a graph $G$ such that every vertex in $G$ has distinct rainbow code. Therefore, we aim to explore the $rvcl(G)$, specifically focusing on regular graphs where all vertices have the same degree. In this research, we focus on classes of vertex transitive graphs, also known as node symmetric graphs, where every pair of vertices is equivalent to some elements of its automorphism group \cite{chiang19952}. Let $n$ be the order of a graph $G$ and $t \geq 3$. An $(n, t)$-regular graph is a graph in which all its vertices have a degree of $n−t$. All graphs of $(n,n-2)$-regular graphs or $2$-regular graphs or cycles, $(n,1)$-regular graphs, $(n,2)$-regular graphs, and $(n,3)$-regular graphs are included in the vertex-transitive graphs. We have determined the locating rainbow connection number for $(n,1)$-regular graphs or complete graphs in Theorem \ref{lemakn}. Therefore, in this paper, we determine the locating rainbow connection number of $2$-regular graphs, $(n,2)$-regular graphs, and $(n,3)$-regular graphs.

\section{Main Results}
The main results are specifically focused on two subsections: the $2$-regular graphs or
cycles, which are extensively discussed in Subsection \ref{Subsec1}, and $(n, t)$-regular graphs in Subsection \ref{Subsec:2}. For simplicity, represent the set $\{n \in \mathbf{Z} \mid x \leq n \leq y\}$ as $[x, y]$.

\subsection{The Locating Rainbow Connection Number of $2$-Regular Graphs}
\label{Subsec1}

The locating rainbow connection number of 2-regular graphs or cycles are closely tied
to the rainbow vertex connection number of cycles. However, not all rainbow vertex
colorings impliy the locating rainbow colorings. Thus, we introduce a new coloring
to fulfill the requirements of locating rainbow colorings on cycles. Some cases require
the values of the rainbow vertex connection number of cycles as presented in Theorem \ref{cn1}.

\begin{teorema}\cite{li2014tight}
Let $C_n$ be a cycle of order $n\geq 3$. Then, 
\begin{center}
	$\begin{array}{ccl}
		rvc(C_n)&=&\left\{
		\begin{array}{ll}
			\lceil\frac{n}{2}\rceil -2, & \hbox{for $n \in \{3,5,9\}$;} \\
			\lceil\frac{n}{2}\rceil -1, & \hbox{for $n \in \{4,6,7,8,10,12,13,15\}$;}\\
			\lceil\frac{n}{2}\rceil, & \hbox{for $n = 14$ or $n \geq 16$.}
		\end{array}
		\right.
	\end{array}$
\end{center}{\tiny }
\label{cn1}

\end{teorema}

In Theorem \ref{thm1}, we see that for relatively small orders, $rvcl(C_n)$ differs from $rvc(C_n)$. For larger orders, $rvcl(C_n)=rvc(C_n)$.

\begin{teorema}
Let $C_n$ be a cycle of order $n\geq 3$. Then,
\begin{center}
	$\begin{array}{ccl}
		rvcl(C_n)&=&\left\{
		\begin{array}{ll}
			3, & \hbox{for $n \in [3,7]$;} \\
			\lceil\frac{n}{2}\rceil -1, & \hbox{for $n \in \{9,15\}$ or $n \in [11,13]$;}\\
			\lceil\frac{n}{2}\rceil, & \hbox{for $n \in \{8, 10, 14\}$ or $n \geq 16$.}
		\end{array}
		\right.
	\end{array}$
\end{center}{\tiny }
\label{thm1}
\end{teorema}

\begin{proof}
Let $n \geq 3$ and $C_n =v_1, v_2, ...,v_n, v_{n+1}$ with $v_{n+1}=v_1$. We consider six cases.

\textbf{Case 1} $n\in [3,7]$

Based on Lemma \ref{lemamemuatliangkaran} and Figure \ref{fig3}, we obtain $rvcl(C_n)=3$ for $n\in [3,7]$.

\begin{figure}[h]
	\centering
	\caption{Locating rainbow $3$-colorings of (a) $C_3$, (b) $C_4$, (c) $C_5$, (d) $C_6$, and (e) $C_7$.}
	\includegraphics[width=4.8in]{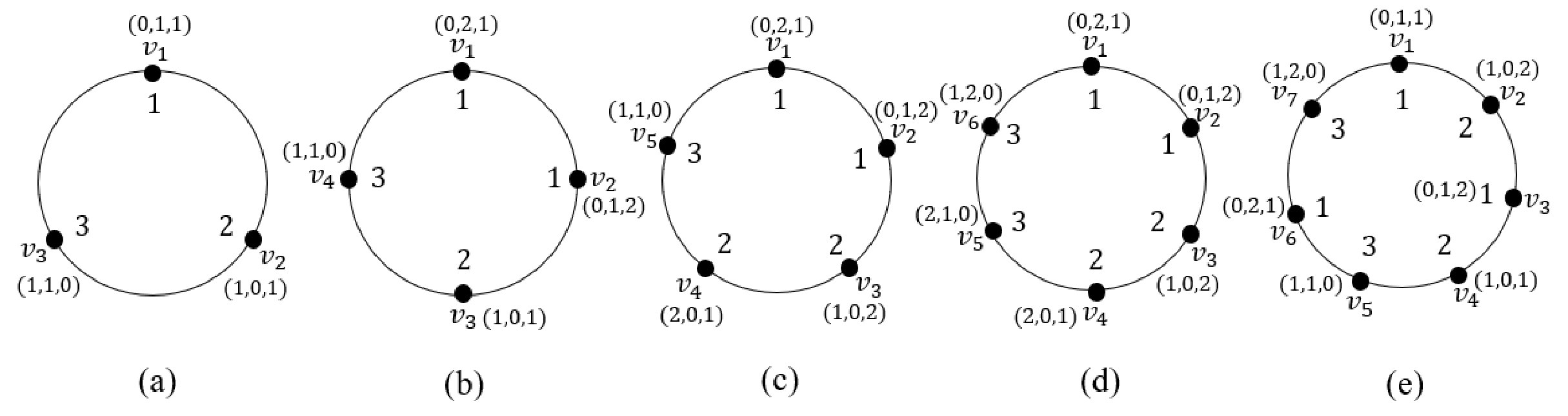}
	\label{fig3}
\end{figure}

\textbf{Case 2} $n=8$
Suppose that $rvcl(C_8 )\geq 3$. Let $c$ be a locating rainbow $3$-coloring of $C_8$. We begin by considering the vertices $v_1$ and $v_5$. Since the vertex coloring of $C_8$ is the rainbow vertex coloring, there exist two possible rainbow vertex $v_1-v_5$ paths, namely $P_1=v_1,v_2,v_3,v_4,v_5$ and $P_2=v_1,v_8,v_7,v_6,v_5$. Without loss of generality, assume that the rainbow vertex path is $P_1$ and $c(v_i)=i-1$ for $i\in[2,4]$. Next, we focus on vertices $v_3$ and $v_7$. There are two possible rainbow vertex paths, i.e., $v_3-v_7$ paths, specifically $P_3=v_3,v_2,v_1,v_8,v_7$, and $P_4=v_3,v_4,v_5,v_6,v_7$. Let the rainbow vertex $v_3-v_7$ path is $P_3$. As $c(v_2 )=1$, it follows that $c(v_1 )=3$ and $c(v_8 )=2$. 
Next, we consider vertices $v_2$ and $v_6$. To ensure the existence of a rainbow vertex path between $v_2$ and $v_6$, we must have $c(v_5 )=1$ or $c(v_7)=1$. If $c(v_5)=1$, then $rc_{\Pi} (v_1 )=rc_{\Pi} (v_4 )$. If $c(v_7)=1$, then $rc_{\Pi} (v_3)=rc_{\Pi} (v_8 )$. For both cases, we get a contradiction. Similarly, in case $v_3-v_7$ path is $P_4$, we observe a similar contradiction. Therefore, we can conclude that $rvcl(C_8 )\geq 4$. By defining a locating rainbow $4$-coloring of $C_8$ as demonstrated in Figure \ref{fig4}(a), we establish that $rvcl(C_8 )=4$.

\textbf{Case 3} $n=9$
Suppose that $rvcl(C_9)=3$. Let $c$ be a locating rainbow $3$-coloring of $C_9$. For $i\in[1,9]$, in order to have a rainbow vertex path between $v_i$ and $v_{((i+4)\pmod{9})+1}$, every three consecutive vertices must be assigned distinct colors. However, this leads to a contradiction as there will be at least two vertices with the same colors and the same rainbow codes. Consequently, we can conclude that $rvcl(C_9)\geq 4$. We show that $rvcl(C_9)= 4$ by defining a locating rainbow $4$-coloring of $C_9$, as shown in Figure \ref{fig4}(b).

\textbf{Case 4} $n=10$
Suppose that $rvcl(C_{10} )=4$. Let $c$ be a locating rainbow $4$-coloring of $C_{10}$. We begin by considering the vertices $v_1$ and $v_6$. Since the vertex coloring of $C_{10}$ is the rainbow vertex coloring, there exist two possible rainbow vertex $v_1-v_6$ paths, namely $P_1=v_1,v_2,v_3,v_4,v_5,v_6$ and $P_2=v_1,v_{10},v_9,v_8,v_7,v_6$. Without loss of generality, assume that the rainbow vertex path is $P_1$ and $c(v_i)=i-1$. Since we have four colors, there must be at least one color used by at least three distinct vertices. Next, we categorize this case into four subcases based on the colors used on at least three distinct vertices in $C_{10}$.
\begin{enumerate}
	\item Color $1$.\\
	Since $c(v_2) = 1$, it follows that $c(v_6) = c(v_9) = 1$. Consequently, $c(v_7) = 2$ and	$c(v_1) = 4$. Therefore, $rc_{\Pi}(v_2) = rc_{\Pi}(v_6) = (0, 1, 2, 1)$, leading to a contradiction.
	\item Color $2$.\\
	To ensure a rainbow vertex path between any two vertices in $C_{10}$, set $c(v_i) =
	c(v_j )$ if $3 \leq d(v_i
	, v_j ) \leq 4$. Consequently, vertices colored with $i \in \{2, 3\}$ are
	$v_{i+1}, v_j, v_k$ for some $j, k \in [1, 10]$, where $i \neq j \neq k$ and $d(v_{i+1}, v_j)$, $d(v_{i+1}, v_k)$,
	and $d(v_j, v_k)$ are in $\{3, 4\}$.
	There are three possible combinations of the remaining two vertices that can
	be colored with $2$. First, if $c(v_3) = c(v_7) = c(v_{10}) = 2$, then vertex $v_6$ and	vertex $v_8$ must be colored with either $1$ or $3$. Consequently, $rc_{\Pi}(v_3) = rc_{\Pi}(v_7)$.
	Second, if $c(v_3) = c(v_6) = c(v_{10}) = 2$, vertices $v_7$ and $v_8$ have to be colored with either $1$ or $3$. In this case, $c(v_9) = 4$ and $c(v_1) = 3$. Therefore, if $c(v_8) = 1$,
	then $rc_{\Pi}(v_6) = rc_{\Pi}(v_{10})$ and if $c(v_8) = 3$, then $rc_{\Pi}(v_5) = rc_{\Pi}(v_9)$. Third,
	$c(v_3) = c(v_6) = c(v_9) = 2$, which implies $c(v_1) = 4$ and $rc_{\Pi}(vv) = rc_{\Pi}(v_4)$, contradiction.
	\item Color $3$.\\
	By employing a similar argument with color $2$, we arrive at a contradiction.
	\item Color $4$.\\
	By employing a similar argument with color $1$, we arrive at a contradiction.
\end{enumerate}

Therefore, $rvcl(C_{10})\geq 5$. To prove that $rvcl(C_{10})\leq 5$, we define a locating rainbow $5$-coloring of $C_{10}$ as shown in Figure \ref{fig4}(c). Thus, $rvcl(C_{10})= 5.$

\begin{figure}[h]
	\centering
	\caption{Rainbow codes of (a) $C_8$, (b) $C_9$, and (c) $C_{10}$.}
	\includegraphics[width=4.6in]{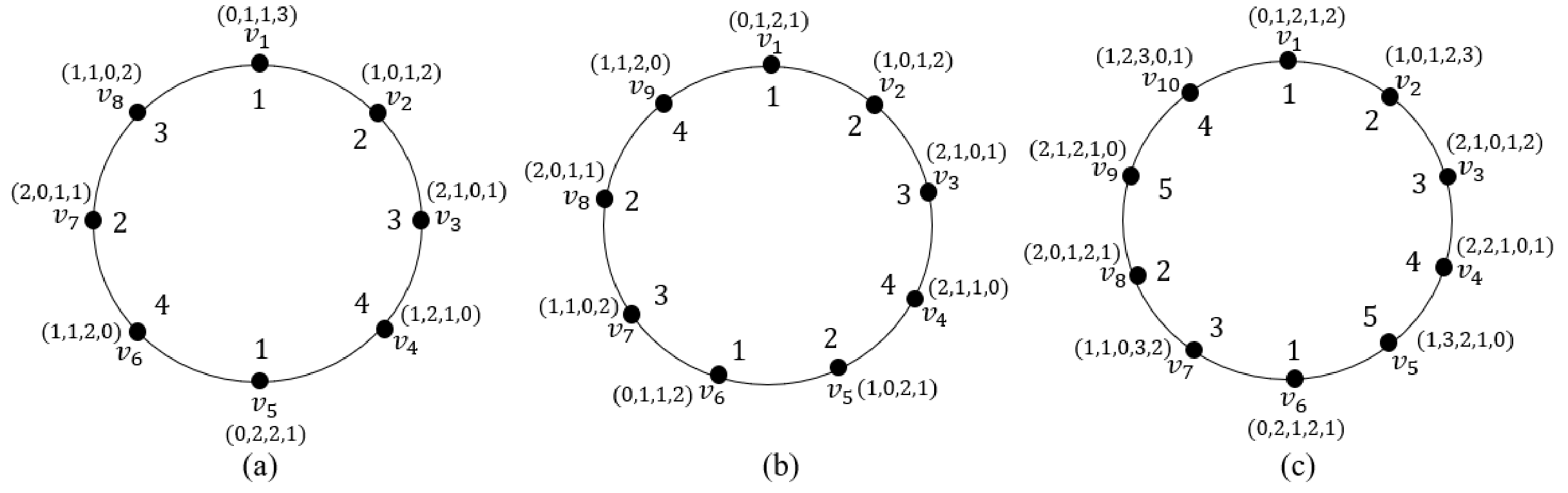}
	
	\label{fig4}
\end{figure}

\textbf{Case 5} $n \in \{11, 12, 13, 15\} $

According to Equation \ref{eq1} and Theorem \ref{cn1}, we have $rvcl(C_n)\geq rvc(C_n)= \lceil\frac{n}{2}\rceil -1$. Next, we show $rvcl(C_n) \leq \lceil\frac{n}{2}\rceil -1$ by defining a locating rainbow $\lceil\frac{n}{2}\rceil -1$-coloring of $C_n$ as shown in Figures \ref{fig5}. Thus, we have $rvcl(C_n)=\lceil\frac{n}{2}\rceil -1$ for $n \in \{11, 12, 13, 15\}$.
\begin{figure}[h]
	\centering
	\caption{Rainbow codes of (a) $C_{11}$, (b) $C_{12}$, (c) $C_{13}$, and $C_{15}$.}
	\includegraphics[width=4.2in]{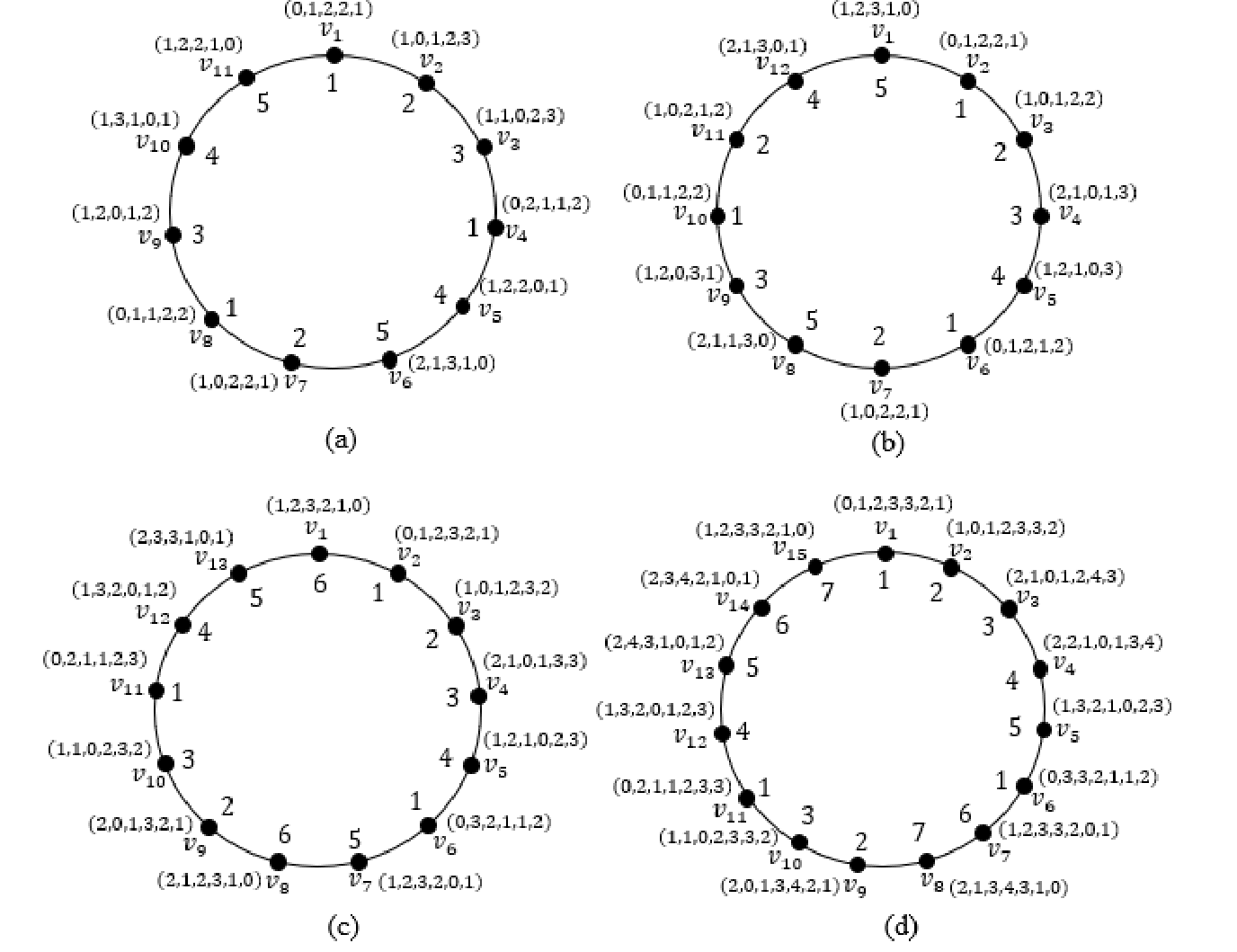}
	\label{fig5}
\end{figure}

\textbf{Case 6} $n=14$ or $n \geq 16$ 

According to Equation \ref{eq1} and Theorem \ref{cn1}, we have $rvcl(C_n)\geq rvc(C_n)= \lceil\frac{n}{2}\rceil$ for $n \geq 16$ or $n =14$. Next, we demonstrate the upper bound by defining  a rainbow vertex coloring $c:V(C_n)\longrightarrow [1,n]$ as follows.
\begin{enumerate}
	\item For odd $n$, we define
	\begin{center}
		$\begin{array}{ccl}
			c(v_i)&=&\left\{
			\begin{array}{ll}
				i~mod ~\lceil\frac{n}{2}\rceil, & \hbox{for $i\in[1,n], i\neq \lceil\frac{n}{2}\rceil$;} \\
				\lceil\frac{n}{2}\rceil, & \hbox{for others.}
			\end{array}
			\right.
		\end{array}$
	\end{center}{\tiny }
	By using vertex coloring above, color $\lceil\frac{n}{2}\rceil$ is only used for $v_{\lceil\frac{n}{2}\rceil}$ and $c(v_i)=c(v_{i+\lceil\frac{n}{2}\rceil})$ for $i \in [1, \lceil\frac{n}{2}\rceil-1]$. Consequently, for any two vertices of $C_n$, there exists a rainbow vertex path. Since $n\neq 0$, $d(v_i,v_{\lceil\frac{n}{2}\rceil}) \neq d(v_{i+\lceil\frac{n}{2}\rceil}, v_{\lceil\frac{n}{2}\rceil})$. Thus, $rc_{\Pi}(v_i)\neq rc_{\Pi}(v_j)$ for distinct $i,j \in [1,n]$.
	
	\item For even $n$, we define
	\begin{center}
		$\begin{array}{ccl}
			c(v_i)&=&\left\{
			\begin{array}{ll}
				\frac{n}{2}-1, & \hbox{for $i=n$;} \\
				\frac{n}{2}, & \hbox{for $i=n-1$;}\\
				i~mod~\frac{n}{2}, & \hbox{for others.}
			\end{array}
			\right.
		\end{array}$
	\end{center}{\tiny }

	Using the vertex coloring above, we have $c(v_i)=c(v_{i+\frac{n}{2}})$ for $n \in [1, \frac{n}{2}-2] \cup [\frac{n}{2}+1, n-2]$, $c(v_{\frac{n}{2}-1})=c(v_n)$, and $c(v_{\frac{n}{2}})=c(v_{n-1})$. Thus, for any two vertices of $C_n$, there exists a rainbow vertex path. Furthermore, we find that $d(v_i,R_{\frac{n}{2}})\neq d(v_{i+\frac{n}{2}}, R_{\frac{n}{2}})$ for $n \in [1, \frac{n}{2}-2]\cup [\frac{n}{2}+1, n-2]$. Additionally, $d(v_{\frac{n}{2}-1}, R_1)\neq d(v_n, R_1)$, and $d(v_{\frac{n}{2}})\neq d(v_{n-1})$. Hence, $rc_{\Pi}(v_i)\neq rc_{\Pi}(v_j)$ for distinct $i,j \in [1,n]$.
\end{enumerate}
Since 	$rvcl(C_n)\geq \lceil\frac{n}{2}\rceil$ and 	$rvcl(C_n)\leq \lceil\frac{n}{2}\rceil$, we conclude that $rvcl(C_n)= \lceil\frac{n}{2}\rceil$
\end{proof}

\subsection{The Locating Rainbow Connection Number of $(n,t)$-Regular Graphs for $t=\{2,3\}$}

In this subsection, we present two main theorems concerning the locating rainbow connection number of $(n,2)$-regular graphs and $(n,2)$-regular graphs. However, before delving into the theorems, we provide Lemma \ref{co2} to aid in the proof process of both theorems. For simplification, we use the term “entry” to refer to the distance from a vertex to a set of colors and the combination formula $C^{r}_{k}$ is expressed as $\frac{r!}{k!(r-k)!}$.

\begin{lemma}
Let $n$ and $t$ be three positive integers with $n \geq 5$, $t=\{2,3\}$, and $r\geq 2$. Let $R_{(n,t)}$ be a regular graph of order $n$ with all vertices have a degree of $n-t$, and let $r$ be the locating rainbow connection number of $R_{(n,t)}$ with $r\geq 3$. Then:
\begin{enumerate}
	\item each rainbow code contains at most $t-1$ of entries $2$;
	\item every color can be used for at most  $1+ \sum_{k=1}^{t-1} C^{r-1}_{t-k}$ vertices;
	\item for every color $w$, the maximum number of vertices $v$, such that $d(v,R_w)=2$ is $t-1$.
\end{enumerate}
\label{co2}
\end{lemma}

\begin{proof}
\begin{enumerate}
	\item Since $diam(R_{(n,t)} )=2$, the graph only contains entries of $0,1$, and $2$. Each vertex in $R_{(n,t)}$ is non-adjacent to exactly $t-1$ other vertices, resulting in the rainbow code possibly containing at most $t-1$ entries of $2$.
	\item Based on the first point in this proof, there are three possible rainbow codes: those without entry $2$, those with $t-2$ entries of $2$, and those with $t-1$ entries of $2$. Hence, every color can be used for at most  $1+ \sum_{k=1}^{t-1} C^{r-1}_{t-k}$ times.
	\item Suppose there are $t$ distinct vertices $v_1, v_2, ..., v_t$ such that $d(v_i,R_w)=2$ for $i\in[1,t]$. Consequently, there exist $v_1', v_2', ..., v_t'$ such that $v_iv_i'\notin E(G)$ and $c(v_i')= 2$ for $i\in[1,t]$. Since $v_iv_j'\in E(G)$ for $i\in[1,t]$ and $i\neq j$, it follows that $d(v_i,R_w)=1$, which is a contradiction. Therefore, for every color $w$, the maximum number of vertices $v$ such that $d(v,R_w)=2$ is $t-1$.

\end{enumerate}
\end{proof}

\subsubsection{The Locating Rainbow Connection Number of $(n,2)$-Regular Graphs}
\label{Subsec:2}

Consider a graph $G$ with order $n\geq4$ and $n$ being even. If all vertices in $G$ are of degree $n-2$, it is termed an $(n,2)$-regular graphs and denoted by $R_{(n,2)}$. On the other hand, $K_n$ represents $(n,1)$-regular graphs. Earlier, in Theorem \ref{lemakn}, we derived the locating rainbow connection number for $(n,1)$-regular graph. Earlier, in Theorem \ref{thm1}, we obtained locating rainbow connection number for $(n,n-2)$-regular graphs, or $2$-regular graphs or cycle graphs. In this section, we focus on determining the locating rainbow connection number of the $R_{(n,2)}$.

\begin{figure}[h]
\centering
\caption{(a) $R_{(12,1)}$, (b) $R_{(12,2)}$}
\includegraphics[width=3.8in]{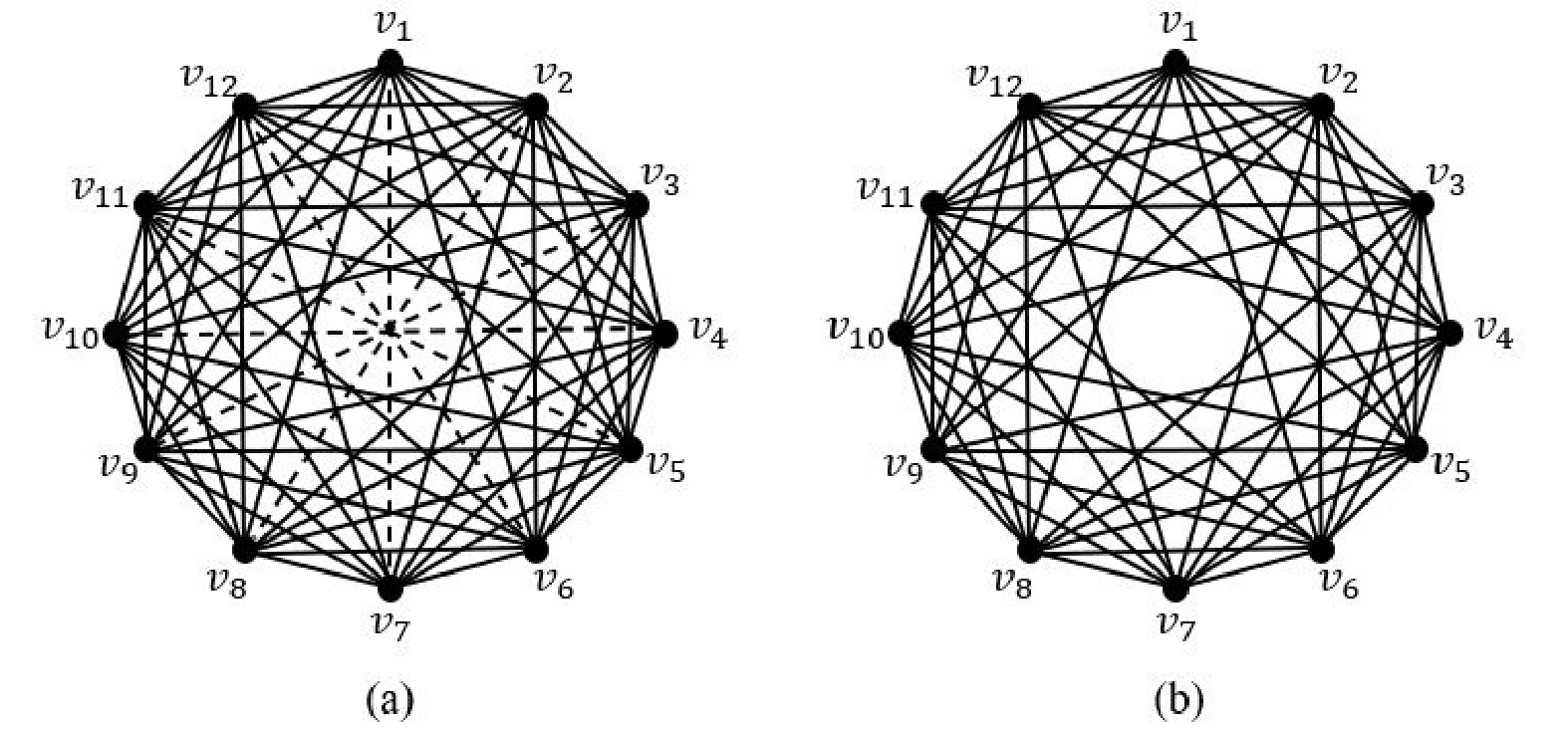}
\label{fig9}
\end{figure}

In actuality, The $(n,2)$-regular graph is obtained by removing $\frac{n}{2}$ disjoint edges from an $(n,1)$-regular graph (See Figure \ref{fig9}). The resulting $(n,2)$-regular graph has the following edges and vertices: $V(R_{(n,2)})=\{v_i | i\in [1,n]\}$, $E(R_{(n,2)})=\{v_iv_j|i,j\in[1,n], i\neq j, j\neq i+\frac{ n}{2}\}$.

\begin{teorema}
Let $n$ be a positive integers with $n \geq 4$, and $n$ be even. If $R_{(n,2)}$ is a regular graph of order $n$ with all vertices of degree $n-2$, then $rvcl(R_{(n,2)})=\frac{n}{2} +1$.
\label{mainteo1}
\end{teorema}

\begin{proof}

Suppose that $rvcl(R_{(n,2)})=\frac{n}{2}$. Based on Lemma \ref{co2}, we have the number of rainbow codes do not contain entry $2$ at most $\frac{n}{2}$ and the number of codes that contain entry $2$ at most $\frac{n}{2}-1$. Thus, the maximum number of different rainbow codes is $n-1$, leading to a contradiction. Therefore, we have $rvcl(R_{(n,2)})\geq \frac{n}{2}+1$.

Furthermore, we demonstrate that $rvcl(R_{(n,2)})\leq \frac{n}{2}+1$ by defining vertex coloring $c: V(R_{(n,2)}) \longrightarrow [1,\frac{n}{2} +1]$ as follows. 
\begin{center}
	$\begin{array}{ccl}
		c(u_i)&=&\left\{
		\begin{array}{ll}
			i, & \hbox{for $i\in[1,\frac{n}{2}+1]$;} \\
			1 , & \hbox{otherwise.}
		\end{array}
		\right.
	\end{array}$
\end{center}

Since $diam(R_{(n,2)})=2$, using $\frac{n}{2}+1$ colors will ensure that for any two distinct vertices $v_i$ and $v_j$ for $i\in[1,n]$, there exists a $v_i-v_j$ rainbow vertex path. Furthermore, we show that all vertices in $R_{(n,2)}$ have distinct rainbow codes by considering the following.
\begin{enumerate}
	\item $c(v_i)\neq c(v_j)$, for distinct $i,j\in[1,\frac{n}{2}+1]$.
	\item $c(v_i)=c(v_j)=1$ for distinct $i,j\in [\frac{n}{2}+2,n] \cup \{v_1\}$, but $d(v_1,R_{\frac{n}{2}+1})=2$ and $d(v_i,R_{\frac{n}{2}+1})=1$ for $i\in [\frac{n}{2}+2,n]$. Besides that, $d(v_i,R_{i-\frac{n}{2}})=2$ and $d(v_i,R_a)=1$ for $a\in[1,\frac{n}{2}]$, $a\neq i-\frac{n}{2}$. 
\end{enumerate}
Thus, $rc_{\Pi}(v_i)\neq rc_{\Pi}(v_j)$ for distinct $i,j \in[1,n]$. Therefore, we have $rvcl(R_{(n,2)})\leq \frac{n}{2}+1$.
\end{proof}

For illustration, in Figure \ref{fig8} we provide  a locating rainbow coloring of the regular graph $R_{(16,2)}$.

\begin{figure}[h]
\centering
\caption{Rainbow codes of $R_{(16,2)}$.}
\includegraphics[width=2.9in]{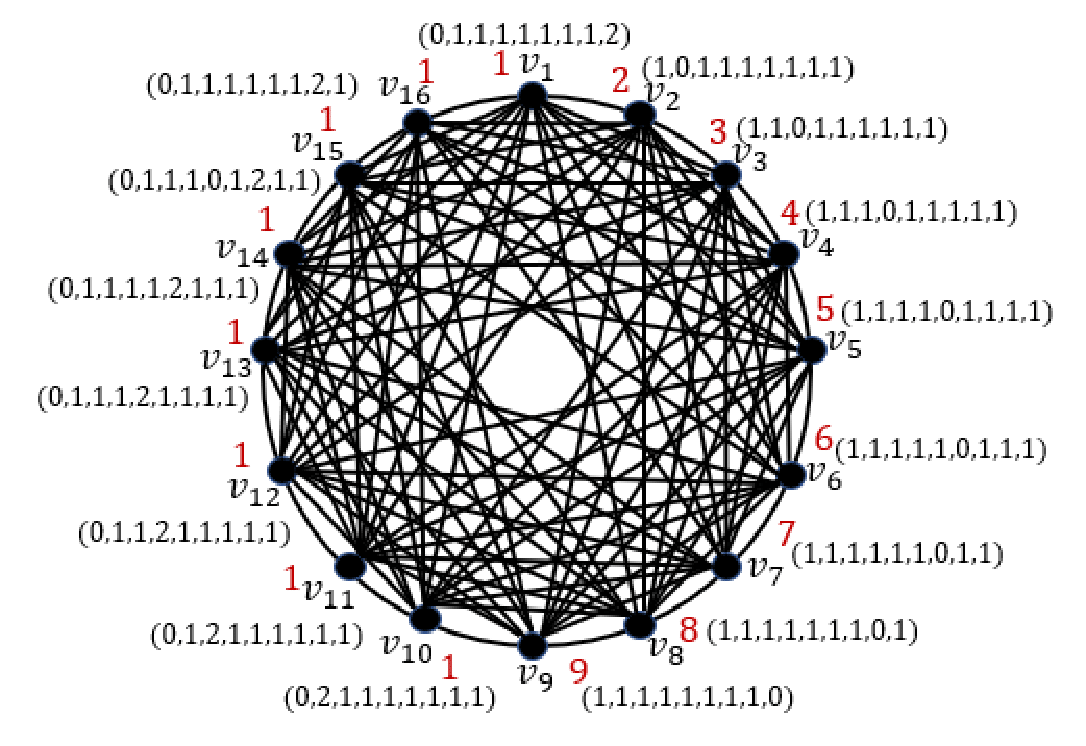}
\label{fig8}
\end{figure}

\subsubsection{The Locating Rainbow Connection Number of $(n,3)$-Regular Graphs}
\label{Subsec:3}

The $(n,3)$-regular graph is obtained by removing a cycle from a $(n,1)$-regular graph (see Figure \ref{fig7}). It is easily observed that the $(n,3)$-regular graph is connected if and only if $n\geq 5$. The $(n,3)$-regular graph has the following edges and vertices.  $V(R_{(n,3)})=\{v_i | i\in [1,n]\}$, $E(R_{(n,3)})=\{v_iv_j|i\in[2,n], j\in[1,n], i\neq j, j= (i+1)~mod~n, j\neq i-1\} \cup \{v_iv_j|i=1, j\in [3,n-1]\}$.
\begin{figure}[h]
\centering
\caption{(a) $R_{(12,1)}$, (b) $R_{(12,3)}$}
\includegraphics[width=4in]{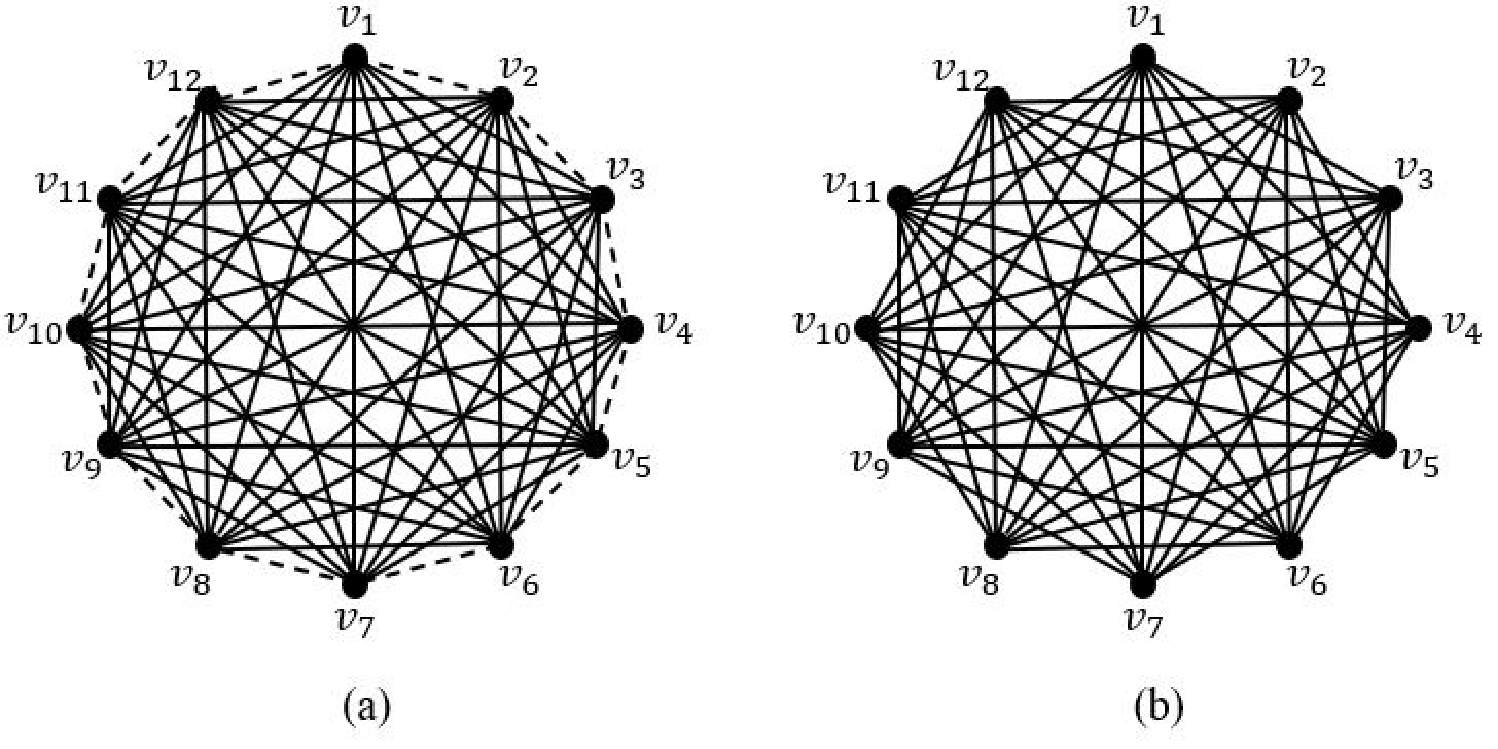}
\label{fig7}
\end{figure}

In Theorem \ref{teo5}, we determine the locating rainbow connection number of the graph $R_{(n,3)}$.

\begin{teorema}
Let $n$ be a positive integer with $n\geq 5$. If $R_{(n,3)}$ is a regular graph of order n with all vertices having a degree of $n-3$, then. 	\begin{center}
	$\begin{array}{ccl}
		rvcl(R_{(n,3)})&=&\left\{
		\begin{array}{ll}
			\lceil\frac{n}{2}\rceil-\lceil\frac{n}{10}\rceil+2, & \hbox{for $n\equiv 2 \pmod{10}$ or $n\equiv 4 \pmod{10}$;}\\
			\lceil\frac{n}{2}\rceil-\lceil\frac{n}{10}\rceil+1, & \hbox{for others.}
		\end{array}
		\right.
	\end{array}$
\end{center}{\tiny }
\label{teo5}
\end{teorema}

\begin{proof}
The proof is partitioned into two cases as outlined below.
\begin{enumerate}
	\item $n\in \{5,6\}$\\
	Based on Lemma \ref{lemamemuatliangkaran}, Theorem \ref{theoremdiam}, and Figure \ref{R56}, we obtain $rvcl(R_{(n,3)})=3$.
	
	\begin{figure}[h]
		\centering
		\caption{$R_{(5,3)}$, $R_{6,3}$.}
		\includegraphics[width=3.3in]{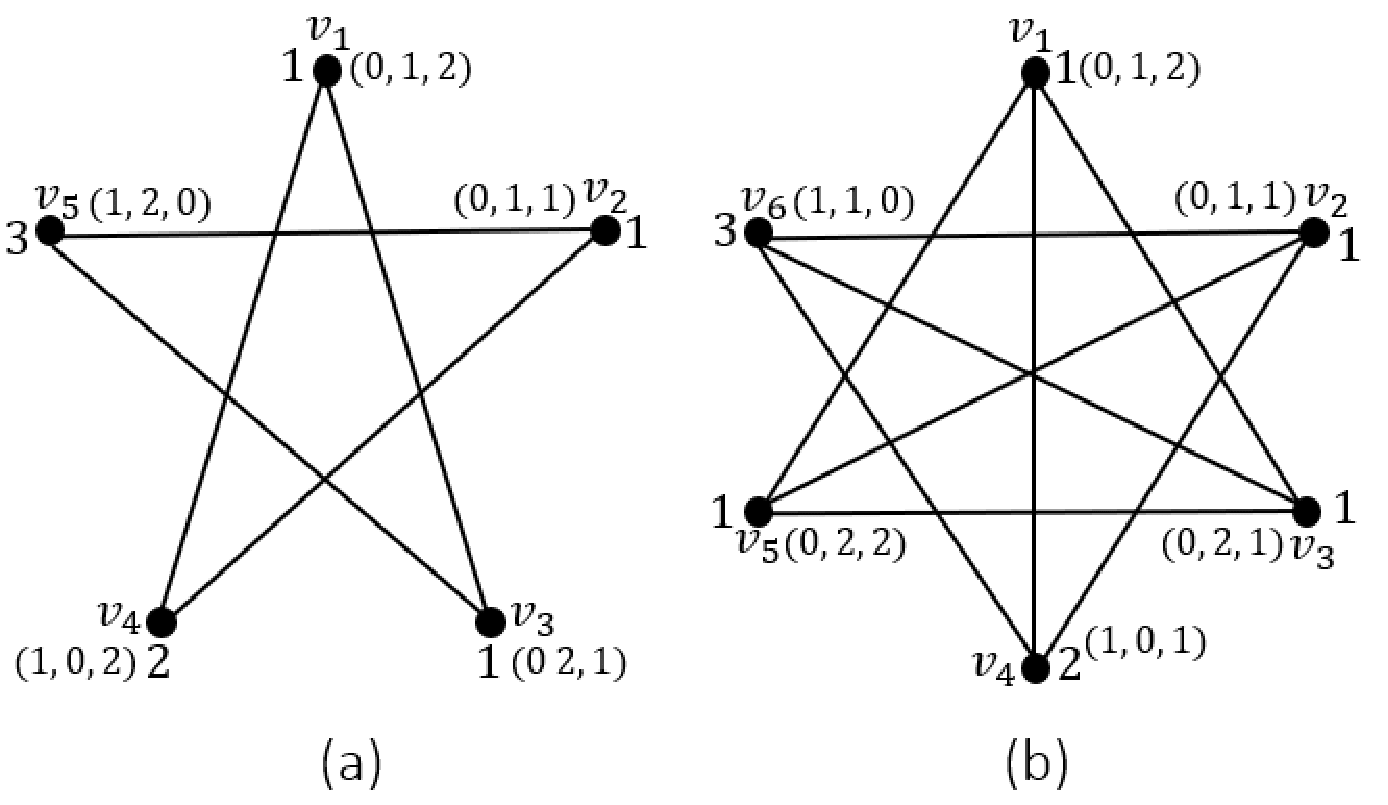}
		\label{R56}
	\end{figure}
	
	\item $n\geq 7$\\
	Suppose $rvcl(R_{(n,3)})= \lceil\frac{n}{2}\rceil-\lceil\frac{n}{10}\rceil+1$ for $n\equiv 2 \pmod{10}$ or $n\equiv 4 \pmod{10}$. Based on Lemma  \ref{co2}, the number of rainbow codes that do not include entry $2$ is at most $\lceil\frac{n}{2}\rceil-\lceil\frac{n}{10}\rceil+1$.
	
	Since $rvcl(R_{(n,3)})<\lceil\frac{n}{2}\rceil$, there must be at least one color used at least three times. Thus, according to Lemma \ref{co2}(3), only $\lceil\frac{n}{2}\rceil-\lceil\frac{n}{10}\rceil$ sets of colors result in an entry of $2$ in a rainbow code. Furthermore, for two pair sets of colors $R_a$ and $R_b$ that produce entry $2$ in a rainbow code, there exist three distinct vertices $u, v, w$ such that $d(u, R_a) = 2$ and $d(u, R_b) = 1$, $d(v, R_a) = 1$ and $d(v, R_b) = 2$, and $d(w, R_a) = d(w, R_b) = 2$.
	
	Therefore, the maximum number of distinct rainbow codes less than $\lceil\frac{n}{2}\rceil-\lceil\frac{n}{10}\rceil+1$, which leads to a contradiction. Similarly, for other values of $n$, contradictions are obtained. Therefore, $rvcl(R_{(n,3)})\geq  \lceil\frac{n}{2}\rceil-\lceil\frac{n}{10}\rceil+2$ for $n\equiv 2 \pmod{10}$ or $n\equiv 4 \pmod{10}$, and $	rvcl(R_{(n,3)})\geq \lceil\frac{n}{2}\rceil-\lceil\frac{n}{10}\rceil+1$ for other values of $n$.
	
	Furthermore, we will show that $rvcl(R_{(n,3)})\leq  \lceil\frac{n}{2}\rceil-\lceil\frac{n}{10}\rceil+2$ for $n\equiv 2 \pmod{10}$ or $n\equiv 2 \pmod{10}$, and $rvcl(R_{(n,3)})\leq 	\lceil\frac{n}{2}\rceil-\lceil\frac{n}{10}\rceil+1$ for other values of $n$. using the following coloring steps. For simplification, $r$ is assumed to represent the number of colors given to a graph $R_{(n,3)}$.
	\begin{enumerate}
		\item  Assign colors $2,3,...,r-1$ to the vertices $v_{4+(i-1)5},v_{6+(i-1)5}$ sequentially for $i\in [1,\frac{r-2}{2}]$ and even $r$.
		\item Assign colors $2,3,...,r-1$ to the vertices $v_{4+(i-1)5},v_{6+(i-1)5}$ sequentially for $i\in [1,\frac{r-3}{2}]$ and even $r$.
		\item In the graph colored with even $r$, set $c(v_n)=r$.
		\item In the graph colored with odd $r$, set $c(v_n)=r$. Furthermore, if $n-(6+(\frac{r-3}{2}-1)5) \in \{3,4\}$, then  $c(v_{n-1})=r-1$, and if $n-(6+(\frac{r-3}{2}-1)5) =5$, then $c(v_{n-2})=r-1$.
		\item 	Finally, color all remaining uncolored vertices with color $1$.
	\end{enumerate}
	
	The diameter of the graph $R_{(n,3)}$ is known to be $2$, indicating that with the given coloring rules, there will always be a rainbow path connecting any two vertices in the graph $R_{(n,3)}$.
	
	Furthermore, the colors $2,3,...,r$ that are assigned only once to the vertices in the graph $R_{(n,3)}$ result in all vertices assigned with these colors having distinct rainbow codes. For color $1$, $c(v_1)=1$ and it is adjacent to the vertices colored with $2,3,...,r$. Additionally, for any two distinct vertices $v_i$ and $v_j$ with $c(v_i)=c(v_j)=1$ for $i,j\in [2,n]$, there is always at least one set of colors $R_a$ for $a\neq 1$, such that $d(v_i,R_a)=2$ and $d(v_j,R_a)=1$. As a result, all vertices in the graph $R_{(n,3)}$ have distinct rainbow codes. Therefore, we obtain $rvcl(R_{(n,3)})= \lceil\frac{n}{2}\rceil-(\lceil\frac{n}{10}\rceil-2)$ for $n\equiv 2 \pmod{10}$ or $n\equiv 2 \pmod{10}$, and $	rvcl(R_{(n,3)})= \lceil\frac{n}{2}\rceil-(\lceil\frac{n}{10}\rceil-1)$ for other values of $n$.
	
\end{enumerate}
\end{proof}

From the Theorem \ref{teo5}, it is obtained that there are three or two graphs with the same locating rainbow connection number. In Figure \ref{R456}, it is shown that $R_{(9,3)}, R_{(10,3)}, R_{(11,3)}$ have the same locating rainbow connection number, which is $5$. Furthermore, Figure \ref{R7-22} indicates that $R_{(12,3)}, R_{(13,3)}$ share the same locating rainbow connection number, which is $6$.

\begin{figure}[h]
\centering
\caption{Rainbow codes of (a) $R_{(9,3)}$, (b) $R_{(10,3)}$, and (c) $R_{(11,3)}$.}
\includegraphics[width=4.7in]{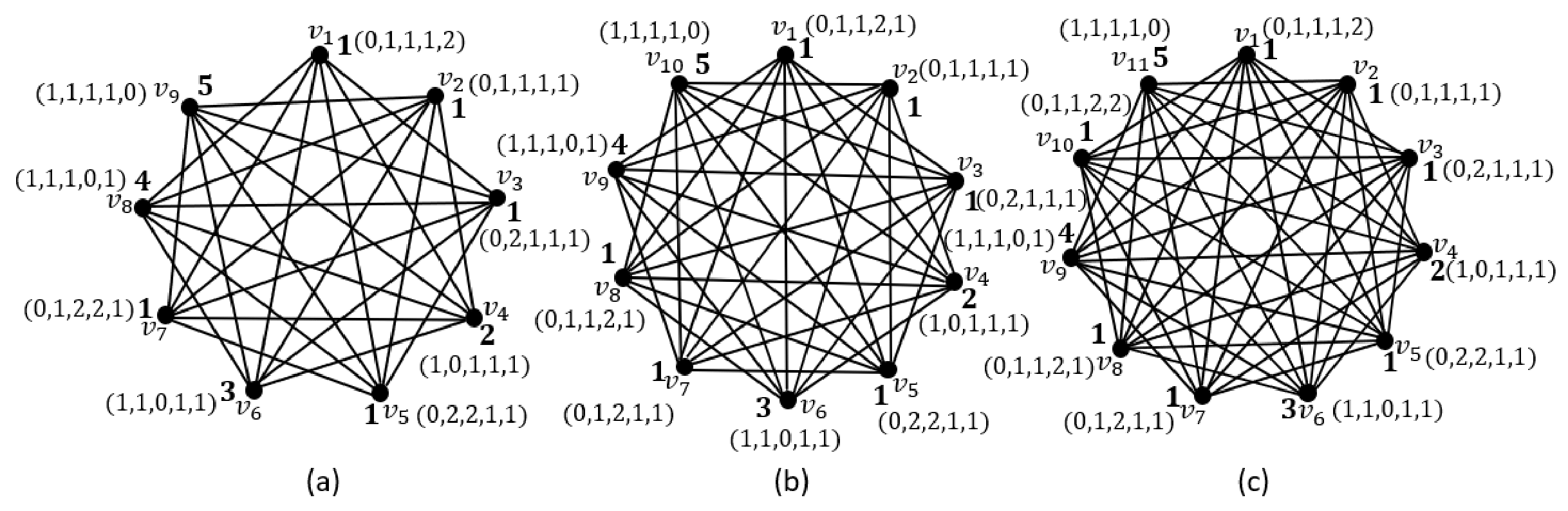}
\label{R456}
\end{figure}

\begin{figure}[h]
\centering
\caption{Rainbow codes of (a) $R_{(12,3)}$ and (b) $R_{(13,3)}$.}
\includegraphics[width=4.5in]{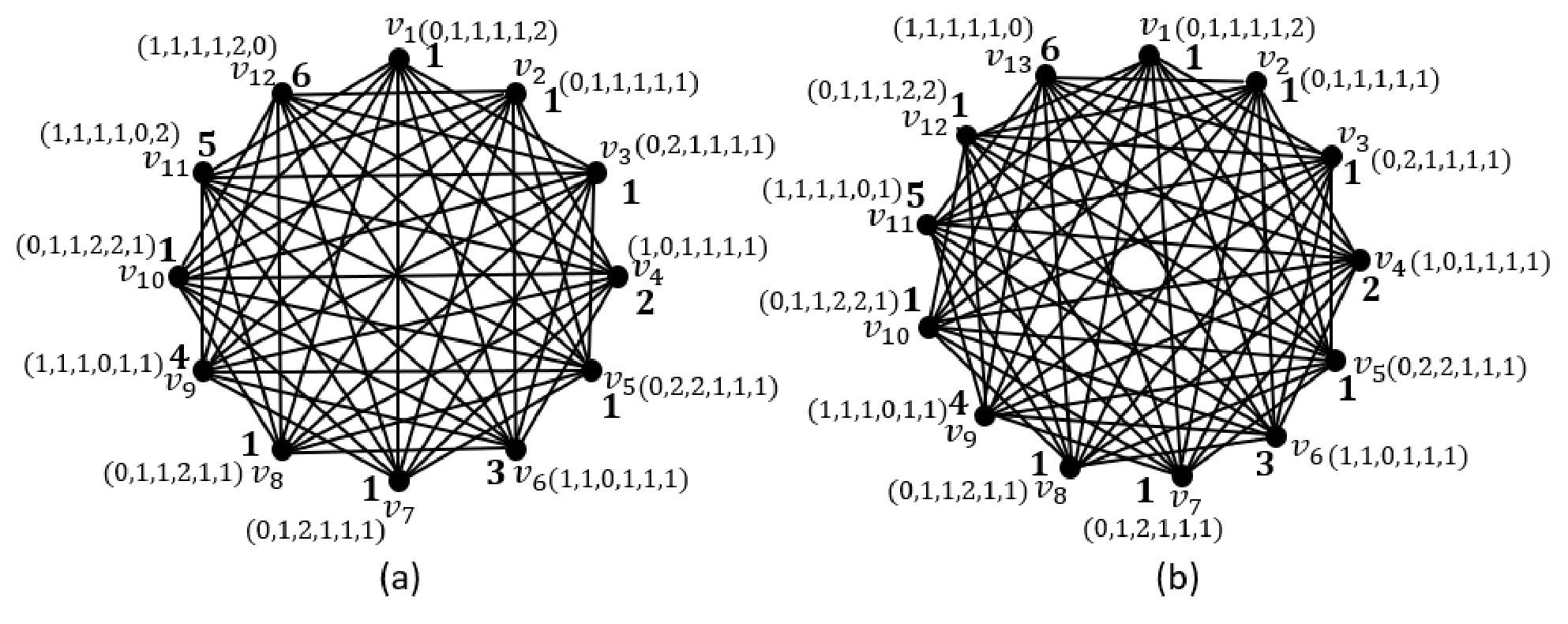}
\label{R7-22}
\end{figure}

All regular graphs discussed in this research are isomorphic within their respective classes. Therefore, to conclude this research, we propose an open problem for future researchers regarding the determination of the locating rainbow connection number for classes of $(n, t)$-regular graphs with $4 \leq t \leq n-3$. This problem is related to determining whether these graph classes form a set of non-isomorphic graphs within their respective classes.

\section*{Acknowledgements}

The authors would like to thank LPDP from the Ministry of Finance Indonesia, the Ministry of Education, Culture, Research, and Technology, and Institut Teknologi Bandung for their support in this research.

\end{document}